\documentclass[12pt]{article}

\usepackage{amsmath, amssymb, amsthm, xcolor,mathtools}
\usepackage{verbatim}
\usepackage{graphicx}

\usepackage[toc]{appendix} 
\usepackage{latexsym}
\usepackage{amsfonts}
\usepackage{amssymb}
\usepackage{amsmath}
\usepackage{mathrsfs}
\usepackage{bbm}
\usepackage{dsfont}
\usepackage{hyperref}
\allowdisplaybreaks

\usepackage[toc]{appendix}

\newtheorem{dfn}{Definition}[section]
\newtheorem{thm}[dfn]{Theorem}
\newtheorem{lmma}[dfn]{Lemma}
\newtheorem{ppsn}[dfn]{Proposition}
\newtheorem{crlre}[dfn]{Corollary}
\newtheorem{xmpl}[dfn]{Example}
\newtheorem{rmrk}[dfn]{Remark}

\newtheorem*{question*}{Question}

\allowdisplaybreaks
\numberwithin{equation}{section}

\newcommand{\bbc}{\mathbb{C}}
\newcommand{\bbz}{\mathbb{Z}}

\newcommand{\bbn}{\mathbb{N}}
\newcommand{\bbr}{\mathbb{R}}

\newcommand{\bbt}{\mathbb{T}}

\date{}

\title{ Is Every Product System Concrete?}
\author{S. Sundar}

\begin{document}

\maketitle

\begin{abstract}
Is every product system of Hilbert spaces over a semigroup $P$  concrete, i.e. isomorphic to the product system of an $E_0$-semigroup over $P$? The answer is no if $P$ is discrete, cancellative and does not embedd in a group. However, we show that the answer is yes for a reasonable class of semigroups.

More precisely, let $P$ be a Borel subsemigroup with non-empty interior of a locally compact, second countable, Hausdorff topological group $G$ such that one of the following conditions hold.
\begin{enumerate}
\item[(1)] The group $G$ is discrete and $P$ is normal in $G$, i.e. $gPg^{-1}= P$ for every $g \in G$.
\item[(2)] The group $G$ is abelian.
\item[(3)] The semigroup $P$ is right Ore, i.e. $PP^{-1}$ is a group, and there exists $a \in P$ such that $a$ is an order unit for $P$.
\end{enumerate}
Then, every product system  of Hilbert spaces over $P$ is isomorphic to the product system of an $E_0$-semigroup over $P$.  

We also extend a result of Liebscher. We  prove that, in the setting of abelian subsemigroups of locally compact groups,  any two measurable structures on a product system differ by a character (possibly non-measurable), which in turn implies that two product systems are isomorphic if and only if they are algebraically isomorphic. 

\end{abstract}
\noindent {\bf AMS Classification No. :} {Primary 46L55, Secondary 46L99}  

\tableofcontents

\section{Introduction}

 In his seminal paper (\cite{Arv_Fock}), Arveson introduced product systems to classify $E_0$-semigroups.  A fundamental result of him states that  the class of product systems over $(0,\infty)$ is in bijective correspondence with the class of $E_0$-semigroups over $(0,\infty)$.  From a purely mathematical standpoint, there is no reason to restrict ourselves
to the semigroup $(0,\infty)$. We can replace $(0,\infty)$ by any semigroup with a measurable structure and  still the notion of product systems and the notion of $E_0$-semigroups make perfect sense. 
 It is then  of  interest to know whether Arveson's bijective correspondence holds in general. The purpose of this paper is to pose this question precisely  and to answer it in the affirmative for a reasonably large class of semigroups.
Based on Arveson's ideas (\cite{Arv}), the author in a joint work with Murugan settled the case   when the semigroup involved is a closed convex cone in $\bbr^d$ (\cite{Murugan_sundar_continuous}), or when the semigroup is a finitely generated subsemigroup of $\bbz^d$ (\cite{Murugan_sundar_discrete_mathsci}).  However, the proofs  depend heavily on the fact that the semigroup involved is abelian and has an order unit. It does not work in the non-commutative situation nor does it work for abelian semigroups that does not have an order unit.  The simplest example of an abelian semigroup that does not have an order unit is the multiplicative semigroup of natural numbers. 

Before we define the objects of interest, let us mention that, in the context of $C^{*}$-algebras, Fowler (\cite{Fowler_2002}), inspired by Pimsner's seminal paper (\cite{Pimsner}) and also by the works of Dinh (\cite{Dinh}, \cite{Dinh_Ktheory}), considered the notion of product systems of $C^{*}$-correspondences over discrete semigroups.  In the discrete case, it is well known that higher rank graphs, or more generally $P$-graphs (\cite{Kumjian_higher}, \cite{Sims_Kumjian}, \cite{Yeend}, \cite{Rennie_Sims}, \cite{Renault_Pgraph}), provide a large supply of product system of $C^{*}$-correspondences  and  the resulting product systems are  product system of Hilbert spaces if the vertex set is singleton. There is a considerable amount of   literature available today that deals with the associated $C^{*}$-algebras and  the non self-adjoint ones.  Also, product and subproduct systems play a vital role  in dilation theory as well, and we  refer the reader to \cite{Skeide_Shalit} for more on this topic. Given this  interest on product systems over general monoids,  the author believes that it is natural  to investigate to what extent Arveson's bijection holds.  In this paper, we restrict our attention to product system of Hilbert spaces, and we do not consider product system of $C^{*}$-correspondences.  Similarly, we restrict attention to $E_0$-semigroups on $B(H)$. The reader interested beyond the case of type $\rm {I}$ factors  may  consult \cite{Skeide4}, \cite{Alexi} and \cite{Srinivasan_Oliver1}, where $1$-parameter $E_0$-semigroups on algebras other than $B(H)$ were considered.

\subsection{Statement of the problem.}
In this subsection, we formulate the question and state the problem precisely. The required definitions  are also collected. 
Let $P$ be a semigroup, and let $\mathcal{B}$ be a $\sigma$-algebra of subsets of $P$. We call $P$ a measurable semigroup if the map \[P \times P \ni (x,y) \to xy \in P\] is measurable. The $\sigma$-algebra that we consider on the product $P \times P$ is the product $\sigma$-algebra. If $(P,\mathcal{B})$ is a standard Borel space, then $P$ is said to be a standard Borel semigroup. A standard Borel semigroup $P$ is said to be discrete if $P$ is countable. 
Let $P$ be a standard Borel semigroup which is fixed until further mention.

\begin{dfn}
Let $E:=\{E(x)\}_{x \in P}$ be a field of non-zero separable Hilbert spaces. Suppose that an associative product is defined on the disjoint union $\displaystyle \coprod_{x \in P}E(x)$. Then, $E$ is called an algebraic product system if 
\begin{enumerate}
\item[(1)] for $x,y \in P$, $u \in E(x)$ and $v \in E(y)$, $uv \in E(xy)$, and
\item[(2)] for $x, y \in P$, there exists a unitary  $U_{x,y}:E(x) \otimes E(y) \to  E(xy)$ such that \[U_{x,y}(u \otimes v)=uv\] for $u \in E(x)$ and $v \in E(y)$.
\end{enumerate}

Let $E:=\{E(x)\}_{x \in P}$ be an algebraic product system. Let $\Gamma$ be a set of sections of $E$. We say the pair $(E,\Gamma)$ is a product system if 
\begin{enumerate}
\item[(1)] $E$ is a measurable field of Hilbert spaces with $\Gamma$ being the space of measurable sections, and 
\item[(2)] for $r,s,t \in \Gamma$, the map 
\[
P \times P \ni (x,y) \to \langle r(x)s(y)|t(xy) \in \bbc\]
is measurable. 
\end{enumerate}
\end{dfn}
If $\big(\{E(x)\}_{x \in P},\Gamma)$ is a product system, we usually suppress $\Gamma$ from the notation, and we simply denote the pair $\big(\{E(x)\}_{x \in P},\Gamma \big)$ by $E$.

\begin{dfn}
Let $E=\{E(x)\}_{x \in P}$ and $F=\{F(x)\}_{x \in P}$ be product systems over $P$. We say that $E$ and $F$ are isomorphic if for each $x \in P$, there exists a unitary operator $\theta_x:E(x) \to F(x)$ such that 
\begin{enumerate}
\item[(1)] for $x,y \in P$, $u \in E(x)$ and $v \in E(y)$, $\theta_{xy}(uv)=\theta_x(u)\theta_y(v)$, and
\item[(2)] $\{\theta_x\}_{x \in P}$ is a measurable field of operators.
\end{enumerate}
\end{dfn}

Next, we recall the notion of $E_0$-semigroups. 
\begin{dfn}
Let $H$ be a separable Hilbert space. An $E_0$-semigroup over $P$ acting on $B(H)$ is a semigroup $\alpha:=\{\alpha_x\}_{x \in P}$ of unital, normal $^*$-endomorphisms of $B(H)$ such that 
for every $A \in B(H)$, the map $P \ni x \to \alpha_x(A) \in B(H)$ is weakly measurable, i.e. for every $\xi,\eta \in H$, the map 
\[
P \ni x \to \langle \alpha_x(A)\xi|\eta \rangle \in \bbc\]
is measurable. 
\end{dfn}

Let $\alpha$ and $\beta$ be $E_0$-semigroups over $P$ on $B(H)$. We say $\beta$ is  a \emph{cocycle perturbation} of $\alpha$, and we write $\beta \sim \alpha$, if there exists a family $U:=\{U_x\}_{x \in P}$ of unitary operators such that 
\begin{enumerate}
\item[(1)] the map $P \ni x \to U_x \in B(H)$ is weakly measurable,
\item[(2)] for $x \in P$, $\beta_x(\cdot)=U_x\alpha_x(\cdot)U_x^*$, and
\item[(3)] for $x, y \in P$, $U_{xy}=U_x\alpha_x(U_y)$.
\end{enumerate}
Note that $\sim$ is an equivalence relation on the set of $E_0$-semigroups on $B(H)$. Let $\alpha$ be an $E_0$-semigroup over $P$ on $B(H)$, and let $\beta$ be an $E_0$-semigroup over $P$ on $B(K)$, where $K$ is another separable Hilbert space. We say that $\alpha$ and $\beta$ are \emph{cocycle conjugate} if there exists a unitary operator $U:H \to K$ such that  $\beta$ is a cocycle perturbation of $\{Ad(U) \circ \alpha_x \circ Ad(U^*)\}_{x \in P}$.

 Let $\alpha:=\{\alpha_x\}_{x \in P}$ be an $E_0$-semigroup over $P$ acting on $B(H)$, where $H$ is a separable Hilbert space. 
For $x \in P$, let $E(x)$ be the intertwining space of $\alpha_x$, i.e. 
\[
E(x):=\{T \in B(H): \alpha_x(A)T=TA \textrm{~~for all $A \in B(H)$}\}.\]
Recall that $E(x)$ is a  separable Hilbert space where the inner product is defined by $\langle S|T \rangle=T^*S$.  Let 
\[
\Gamma:=\{f:P \to B(H): f(x) \in E(x) \textrm{~~and $f$ is weakly measurable}\}.\] 
Then, $E_\alpha:=\big(\{E(x)\}_{x \in P},\Gamma \big)$ is a product system over $P$. The product rule on $E_\alpha$ is the usual composition of operators. 
The product system $E_\alpha$ is called \emph{the product system associated with $\alpha$}. Arguing as in the $1$-parameter case, we can prove that the association 
\[
\alpha \to E_\alpha\] is injective. More details for the case of a cone can be found in \cite{Murugan_Sundar}.

\begin{dfn}
A product system over $P$ is said to be \emph{concrete} if it is isomorphic to the product system associated with an $E_0$-semigroup over $P$. 
\end{dfn}

We can now pose the question of interest.

\begin{question*}
 Let $P$ be a standard Borel semigroup. Is every product system over $P$ concrete? 
\end{question*}

It is elementary to prove that a product system is concrete if and only if it has an essential representation on a separable Hilbert space (see \cite{Arveson} for a proof in the $1$-parameter case). 
Let $E$ be a product system over $P$, and let $H$ be a Hilbert space. A  map $\phi:\displaystyle \coprod_{x \in P}E(x) \to B(H)$ is called a representation of $E$ on $H$ if 
\begin{enumerate}
\item[(1)] for $x \in P$, $u,v \in E(x)$, $\phi(v)^*\phi(u)=\langle u|v \rangle$, and
\item[(2)] for $x,y \in P$, $u \in E(x)$ and $v \in E(y)$, $\phi(uv)=\phi(u)\phi(v)$. 
\end{enumerate} 
Let $\phi$ be a representation of $E$ on  a Hilbert space $H$. We say that $\phi$ is measurable if for any measurable section $s$ of $E$, the map 
\[
P \ni x \to \langle \phi(s(x))\xi|\eta \rangle \in \bbc\]
is measurable for every $\xi,\eta \in H$. A representation $\phi$ is said to be essential if for every $x \in P$, $[\phi(E(x))H]=H$ (for a subset $\mathcal{S} \subset B(H)$, $[\mathcal{S}H]$ denotes the closure of the linear span of $\{T\xi: T \in \mathcal{S}, \xi \in H\}$). 

We can rephrase our question as follows.

\begin{question*}
 Let $P$ be a standard Borel semigroup, and let  $E$ be a product system over $P$. Does there exist a separable Hilbert space $H$ and a representation $\phi$ of $E$ on $H$ such that $\phi$ is essential and $\phi$ is measurable? 
 \end{question*}

The answer is no in general. We show that if $P$ is a discrete, cancellative semigroup that does not embedd in a group, then there is  a product system over $P$ that is not concrete. It may be too ambitious to expect a positive answer even for  subsemigroups of groups.  Nevertheless, we  show that the answer is yes for a reasonably good class of semigroups of locally compact groups. 

Before we proceed further, let us mention the examples of semigroups for which the above question is resolved in the literature.  It is known that every product system over a semigroup $P$ has a measurable, essential representation on a separable Hilbert space  if  eithier $P=(0,\infty)$ (\cite{Arv_Fock4}, \cite{Arv}, \cite{Skeide}, \cite{Liebscher}), or if $P$ is a closed, pointed cone in $\bbr^d$ (\cite{Murugan_sundar_continuous}), or if $P$ is a finitely generated subsemigroup of $\bbz^d$ (\cite{Murugan_sundar_discrete_mathsci}). The classical case is when $P=(0,\infty)$ which started the subject. The first proof for the case $P=(0,\infty)$ was due to Arveson (\cite{Arv_Fock4}). Alternate proofs by Liebscher (\cite{Liebscher}), Skeide (\cite{Skeide}, \cite{Skeide1}) and by Arveson (\cite{Arv}) himself are now available in the literature.   We also wish to mention  that if $P$ is discrete and if each fibre of $E$ is finite dimensional, then our problem amounts to asking whether the Cuntz-Pimsner algebra $\mathcal{O}_E$ defined by Fowler in \cite{Fowler_2002} is non-zero. A positive solution in the case when  $P$ is a discrete Ore semigroup   and when the fibres of $E$ are finite dimensional can be found in \cite{Szymanski} or in \cite{Meyer} where the  authors work in the more general framework of  regular/proper $C^{*}$-correspondences. 

\subsection{Main results.}
We next summarise  the results obtained and give a brief overview of the ideas involved in proving them.  The main result of this paper is the following. 
\begin{thm}
\label{main}
Let $G$ be a locally compact, second countable, Hausdorff topological group, and let $P \subset G$ be a Borel subsemigroup with non-empty interior. Assume that  
\begin{enumerate}
\item[(C1)] either the group $G$ is discrete and $P$ is normal in $G$, i.e. $gPg^{-1}=P$ for every $g \in G$, or
\item[(C2)] the group $G$ is abelian, or
\item[(C3)] the set $PP^{-1}$ is a subgroup of $G$, and there exists $a \in P$ such that $a$ is an order unit for $P$, i.e. $\bigcup_{n=1}^{\infty}Pa^{-n}=PP^{-1}$.
\end{enumerate}
Then, every product system  of  Hilbert spaces over $P$ has a representation $\phi$ on a separable Hilbert space such that $\phi$ is measurable and $\phi$ is essential. In particluar, every product system over $P$ is concrete. 
\end{thm} 
   The main ideas and the sequence of steps involved in the proof are described below.

For a product system $E$ over $(0,\infty)$, Skeide presented a construction in \cite{Skeide} (essentially an induced construction) that allows us to define an essential representation of $E$ starting from an essential representation of $E|_\bbn$. Here, $E|_\bbn$ stands for the product system restricted to $\bbn$. The conclusion in the $1$-parameter case is immediate as it is elementary to see that every product system over $\bbn$ has an essential representation.
  We generalise this trick  which is one 
of the key steps in the proof of Thm. \ref{main}. 

Let  $G$ be a locally compact, second countable, Hausdorff topological group, and let $P \subset G$ be Borel subsemigroup with non-empty interior. Assume that $PP^{-1}$ is a group. Suppose $E$ is a product system over $P$, and $Q$ is a  Borel subsemigroup of $P$. We denote the product system restricted to $Q$ by $E|_Q$. Then, under the assumption that $QQ^{-1}$ is a closed subgroup and $Q$ is cofinal, i.e. $PQ^{-1}=PP^{-1}$, we show that we can induce an essential representation of $E|_Q$ to an essential representation of $E$. Thus, if $E|_Q$ is concrete, then $E$ is concrete.  It is now clear that  under (C3), Thm. \ref{main} follows. To tackle the case when $G$ is abelian, we show the existence of a countable cofinal semigroup $Q$ such that $Q-Q$ is closed. This reduces proving Thm. \ref{main} to the case when $G$ is discrete. 

The second crucial step that allows to prove Thm. \ref{main} under the assumption that $G$ is discrete is the following theorem.

\begin{thm}
\label{inductive limit}
Let $P$ be a countable, discrete  semigroup. Assume that $P$ is left reversible, i.e. for $a,b \in P$, $aP \cap bP \neq \emptyset$. Suppose $(P_n)$ is a sequence of subsemigroups such that $P_n \nearrow P$. Let $E$ be a product system over $P$, and let $E_n$ be the restriction of $E$ to $P_n$. If $E_n$ is concrete for every $n$, then $E$ is concrete. 
\end{thm}
The proof of Thm. \ref{inductive limit} is nice application of an ultraproduct construction.  With the notation of Thm. \ref{inductive limit}, let $\phi_n$ be an essential representation of $E_n$ on a separable Hilbert space $H_n$.   Choose a bounded linear functional $\omega:\ell^{\infty}(\bbn) \to \bbc$ that vanishes on $c_0(\bbn)$. Let $H_\infty$ be the set of norm bounded sequences $(\xi_n)_n$ with $\xi_n \in H_n$. The formula 
\[
\langle (\xi_n)_n|(\eta_n)_n\rangle: =\omega ((\langle \xi_n|\eta_n \rangle)_n)\]
defines a semi-definite inner product on $H_\infty$ whose completion we still denote by $H_\infty$.   Morever, the formula 
\[
\phi(u)((\xi_n)_n)=(\phi_n(u)\xi_n)_n\]
defines a representation of $E$ on $H_\infty$. We show that while $H_\infty$ may not be separable, it has a separable reducing subspace on which $\phi$ is essential.

With Thm. \ref{inductive limit} in hand, the proof of Thm. \ref{main} is completed as follows. As already observed, we are only left with proving Thm. \ref{main} when $G$ is discrete and $P$ is a normal subsemigroup. In this case, we show that we can write $P$ as an increasing union $ \bigcup_{n \geq 1}P_n$ where each $P_n$ is a normal subsemigroup for which Condition (C3) of Thm. \ref{main} is satisfied. As we have already seen  that Thm. \ref{main} holds under (C3), we can apply Thm. \ref{inductive limit} to finish the proof.

We end our paper with a non-example and by extending a result of Liebscher. 
For a non-example, let $P$ be a discrete, countable, cancellative semigroup that does not embedd in a group.  It was  Malcev (\cite{Malcev}) who first showed the existence of such a semigroup.  Let $E$ be the product system of the  CCR flow associated with the left regular representation of $P$ on $\ell^2(P)$. The opposite $E^{op}$ of $E$ is then a product system over the opposite semigroup $P^{op}$, and we show that $E^{op}$ is not concrete. We conclude by extending a result of Liebscher, by proving that, in the setting of abelian subsemigroups of locally compact abelian groups, two product systems are isomorphic if and only if they are algebraically isomorphic. 
 The proof is a nice application of our main result and is inspired by a trick  that Arveson  (\cite{Arv_Spectral1990}) exploited to show that every $1$-parameter $E_0$-semigroup admits an extension to a group of automorphisms.   It is irresistible for me not to record this trick, and I believe that the proof has some elegance attached to it. 


\textit{Notation:} For us, $\bbn=\{1,2,\cdots \}$ and $\bbn_0:=\bbn \cup \{0\}$. 


\section{Induced representations}

As mentioned earlier, the inspiration for this section comes from Skeide's work (\cite{Skeide}). Liebscher's paper (\cite{Liebscher}) is yet another motivation for the author especially  to consider the Hilbert space $H_{[x]}$ that appears below. 

For the rest of this section, let $G$ be an arbitrary locally compact, second countable, Hausdorff  topological group, and let $P \subset G$ be a Borel subsemigroup, i.e. $P$ is  a Borel set which is also a subsemigroup.  We assume that $P$ is right Ore, i.e. $PP^{-1}$ is a subgroup. This ensures that $P$ is left reversible, i.e. given $a,b \in P$, $aP \cap bP \neq \emptyset$.  
 We  assume that $P$  has non-empty interior. Since $P$ has non-empty interior, $PP^{-1}=Int(P)Int(P)^{-1}$.  Consequently, $PP^{-1}$ is an open subgroup. Thus, we may assume $PP^{-1}=G$. 

 For $x, y \in P$, we say $x \leq y$ if there exists $a \in P$ such that $y=xa$.  Since $P$ is left reversible, given $x,y \in P$, there exists $z \in P$ such that $z \geq x$ and $z \geq y$. 
A subset $F \subset P$ is said to be cofinal if given $x \in P$, there exists $f \in F$ such that $x \leq f$. Since $PP^{-1}=G$, if $F \subset P$ is cofinal, then  it is clear that given $x \in G$, there exists $f \in F$
such that $f \ge x$, i.e. $x^{-1}f \in P$.

\begin{ppsn}
There exists a sequence $(s_n) \in Int(P)$ such that $\{s_n: n \in \bbn\}$ is cofinal  and $s_n \leq s_{n+1}$ for every $n$. 
\end{ppsn}
\textit{Proof.} By assumption, $\{Int(P) a^{-1}\}_{a \in Int(P)}$ is an open cover of $G$. Since $G$ is second countable,  there exists a sequence $(t_n)$ in $Int(P)$ such that $G=\bigcup_{n=1}^{\infty}Int(P) t_n^{-1}$, i.e $\{t_n:n \in \bbn\}$ is cofinal. Set $s_1=t_1$ and define $s_n$ inductively as follows: for $n \in \bbn$, pick $s_n \in Int(P)$ such that $s_n \geq t_n$  and $s_n \geq s_{n-1}$. 
This is possible since $Int(P)$ is left reversible. By definition, $(s_n)$ is increasing.  Since $s_n \geq t_n$ and $\{t_n: n \in \bbn\}$ is cofinal, it follows that $\{s_n: n \geq 1\}$ is cofinal. \hfill $\Box$ 

 For a closed subgroup $K$, we denote the homogeneous space of left cosets of $K$ by $G/K$. For $x \in G$, we denote the left coset $xK$ by $[x]$. 

\begin{ppsn}
Let $Q \subset P$ be a Borel subsemigroup such that $Q$ is cofinal. Suppose $K:=QQ^{-1}$ is a closed subgroup of $G$. 
Then, the quotient map $\pi:G \to G/K$ has a Borel cross section $\kappa:G/K \to G$ such that $\kappa(G/K) \subset P$.
\end{ppsn}
\begin{proof}
 Let $(s_n) \in Int(P)$ be a cofinal sequence and for each $n \in \bbn$, let $t_n \in Q$ be such that $t_n \geq s_n$. Then, $\displaystyle \bigcup_{n=1}^{\infty}Pt_n^{-1}=G$.  We can also ensure that $t_n \leq t_{n+1}$ for every $n$. Set $A_1:=Pt_{1}^{-1}$, and let $A_k:=Pt_{k}^{-1} \backslash Pt_{k-1}^{-1}$ for $k \geq 2$. Then, 
$G=\displaystyle \coprod_{k=1}^{\infty}A_k$. 

Define a map $\phi:G \to P$ by setting $\phi(x)=xt_k$ if $x \in A_k$. Then, $\phi$ is a Borel map. 
Let $\kappa_0: G/K \to G$ be a Borel cross section. For the existence of such a cross section, we refer the reader to Thm. 3.4.1 of \cite{Arveson_invitation}. 
Define $\kappa:=\phi \circ \kappa_0$. Then, $\kappa$ is a Borel cross section of the quotient map $\pi:G \to G/K$ such that the range of $\kappa$ is contained in $P$. The proof is complete. \end{proof}

Let $E:=\{E(x)\}_{x \in P}$ be a product system over $P$. 
For $x,y \in P$, denote the unitary operator $E(x) \otimes E(y) \ni u \otimes v  \to uv \in E(xy)$ by $U_{x,y}$.
Let $Q \subset P$ be a cofinal Borel subsemigroup such that $K:= QQ^{-1}$ is a closed subgroup of $G$. 
Let $\sigma$ be an essential representation of $E|_Q$ on a separable Hilbert space $H$. Further, assume
that $\sigma$ is measurable. 
For $a \in Q$, let $\sigma_a:E(a) \otimes H \to H$ be the unitary operator defined by 
\[
\sigma_a(u \otimes \xi)=\sigma(u)\xi.\]
For $[x] \in G/K $, define \[
\Delta([x]):=\coprod_{y \in P, [y]=  [x]}E(y)\otimes H.\]
Note that since $Q$ is cofinal, given $x \in G$, there exists $y \in P$ such that $[y]=[x]$. Hence, $\Delta([x)]$ is non-empty. 
For $y,z \in [x] \cap P$, $u \in E(y) \otimes H$ and $v \in E(z) \otimes H$, we say $u \sim v$ if there exists $a,b \in Q$ such that $ya=zb$ and 
\[
(U_{y,a} \otimes 1)(1 \otimes \sigma_a^*)u=(U_{z,b} \otimes 1)(1 \otimes \sigma_b^*)v.\]

\begin{ppsn}
\label{equivalencegoat}
With the above notation, the relation $\sim$ on $\Delta([x])$ is an equivalence relation for every $[x] \in G/K$. 
\end{ppsn}
\textit{Proof.}
  For a Hilbert space $\mathcal{L}$, we denote the identity operator on $\mathcal{L}$ by $1_{\mathcal{L}}$.
 Let $[x] \in G/K$. Let $y,z \in [x] \cap P$. Choose $a,b \in Q$ such that $ya=zb$. Define 
\[
W_{z,y}:=(1_{E(z)} \otimes \sigma_b)(U_{z,b} \otimes 1_H)^*(U_{y,a} \otimes 1_H)(1_{E(y)} \otimes \sigma_a)^{*}.\]
We claim that $W_{z,y}$ does not depend on $a,b$.

Suppose $\alpha,\beta \in Q$ are such that $y\alpha=z\beta$. Since $a^{-1}\alpha \in QQ^{-1}$, choose $c,d \in Q$ such that $a^{-1}\alpha=cd^{-1}$. Note that \[cd^{-1}=a^{-1}\alpha=(ya)^{-1}(y\alpha)=(zb)^{-1}z\beta=b^{-1}\beta.\] Hence, 
$ac=\alpha d$ and $bc=\beta d$. 
Using the fact that multiplication on $E$ is associative and the fact that $\sigma$ is a representation of $E|_Q$, we  calculate as follows to observe that 
\begin{align*}
&(U_{y,ac} \otimes 1_H)(1_{E(y)}\otimes \sigma_{ac}^{*})\\&=(U_{y,ac} \otimes 1_H)(1_{E(y)}\otimes \big(U_{a,c} \otimes 1_H)(1_{E(a)}\otimes \sigma_c^*)\sigma_a^*\big)\\
&=(U_{y,ac}\otimes 1_H)(1_{E(y)}\otimes U_{a,c} \otimes 1_H)(1_{E(y)}\otimes (1_{E(a)} \otimes \sigma_c^*)\sigma_a^*)\\
&=(U_{ya,c} \otimes 1_H)(U_{y,a} \otimes 1_{E(c)} \otimes 1_H)(1_{E(y)}\otimes 1_{E(a)} \otimes \sigma_c^*)(1_{E(y)} \otimes \sigma_a^*)\\
&=(U_{ya,c} \otimes 1_H)(U_{y,a}\otimes \sigma_c^*)(1_{E(y)}\otimes \sigma_a^*).
\end{align*}
Hence, 
\begin{equation}
\label{goat}
(U_{y,ac} \otimes 1_H)(1_{E(y)}\otimes \sigma_{ac}^{*})=(U_{ya,c} \otimes 1_H)(U_{y,a}\otimes \sigma_c^*)(1_{E(y)}\otimes \sigma_a^*).
\end{equation}
Similarly, 
\begin{equation}
\label{goat1}
(U_{z,bc} \otimes 1_H)(1_{E(z)}\otimes \sigma_{bc}^{*})=(U_{zb,c} \otimes 1_H)(U_{z,b}\otimes \sigma_c^*)(1_{E(z)}\otimes \sigma_b^*).
\end{equation}
Eq. \ref{goat}, Eq. \ref{goat1} and the fact that $ya=zb$ imply that 
\begin{equation}
\label{goat2}
(1_{E(z)} \otimes \sigma_b)(U_{z,b} \otimes 1_H)^*(U_{y,a} \otimes 1_H)(1_{E(y)} \otimes \sigma_a)^{*}
=(1_{E(z)} \otimes \sigma_{bc})(U_{z,bc} \otimes 1_H)^*(U_{y,ac} \otimes 1_H)(1_{E(y)} \otimes \sigma_{ac})^{*}
\end{equation}
Similarly, 
\begin{equation}
\label{goat3}
(1_{E(z)} \otimes \sigma_\beta)(U_{z,\beta} \otimes 1_H)^*(U_{y,\alpha} \otimes 1_H)(1_{E(y)} \otimes \sigma_\alpha)^{*}=(1_{E(z)} \otimes \sigma_{\beta d})(U_{z,\beta d} \otimes 1_H)^*(U_{y,\alpha d} \otimes 1_H)(1_{E(y)} \otimes \sigma_{\alpha d})^{*}
\end{equation}
Since $ac=\alpha d$ and $bc=\beta d$, it follows from Eq. \ref{goat2} and Eq. \ref{goat3} that the operator $W_{z,y}$ is well defined. It is clear that $W_{z,y}$ is a unitary operator.

We claim that the family of operators $\{W_{z,y}\}_{z,y \in [x] \cap P}$ satisfies the following properties. 
\begin{enumerate}
\item[(C1)] For $z \in [x]$, $W_{z,z}=1$,
\item[(C2)] for $z, y \in [x]$, $W_{z,y}^{*}=W_{y,z}$, and 
\item[(C3)] for $z_1,z_2,z_3 \in [x]$,
\[
W_{z_1,z_2}W_{z_2,z_3}=W_{z_1,z_3}.\]
\end{enumerate}

Condition (C1) and (C2) follow from definition. To see (C3), let 
$z_1,z_2,z_3 \in [x] \cap P$ be given. Choose $a,b,c,d \in Q$ such that $z_1a=z_2b$ and $z_2c=z_3d$. We can ensure $b=c$. Otherwise, choose $\alpha,\beta \in Q$ such that $b\alpha=c\beta$, and replace $a,b,c,d$  by $a\alpha,b\alpha,c\beta, d\beta$ respectively. Then, it is clear from the definition that 
\[
W_{z_1,z_2}W_{z_2,z_3}=W_{z_1,z_3}.\]
This proves the claim. It is now immediate that $\sim$ is an equivalence relation. \hfill $\Box$. 

Denote the set of equivalence classes of $\Delta([x])$ by $H_{[x]}$. Then, $H_{[x]}$ has a natural vector space structure and also has an inner product that makes $H_{[x]}$ a separable Hilbert space. 
To see this, pick $y \in [x] \cap P$. Let $i:E(y) \otimes H \to \Delta([x])$ be the `natural inclusion' map, and let $q:\Delta([x]) \to H_{[x]}$ be the quotient map. Then, $q \circ i:E(y) \otimes H \to H_{[x]}$ is a bijective map. We impose the vector space structure and the inner product structure on $H_{[x]}$ after identifying it with $E(y) \otimes H$ via the map $q \circ i$.  Then, the vector space structure as 
well as the inner product is independent of the chosen representative.

For $a \in P$, $u \in E(a)$ and $[x] \in G/K$, let  $\theta([x], u):H_{[x]} \to H_{[ax]}$ be the bounded operator  defined by 
\[
\theta([x], u)[v \otimes \xi]=[uv \otimes \xi].\]
Note that for $a,b \in P$, $u \in E(a)$ and $v \in E(b)$, 
\begin{equation}
\label{associativity}
\theta([bx],u)\theta([x],v)=\theta([x],uv).
\end{equation}
Also, for $a \in P$, $u,v \in E(a)$,
\begin{equation}
\label{inner product preserving}
\theta([x],v)^*\theta([x],u)=\langle u|v \rangle.\end{equation}

\begin{ppsn}
\label{essential}
Let $a \in P$, and  let $\{e_i\}_{i=1}^{d}$ be an orthonormal basis for $E(a)$, where $d:=\dim E(a)$. Then, for every $[x] \in G/K$, 
\[
\sum_{i=1}^{d}\theta([x],e_i)\theta([x],e_i)^*=1.\]
\end{ppsn}
\textit{Proof.} Fix $[x] \in G/K$.  We can assume $x \in P$. Denote $\theta([x],e_i)$ by $\theta(e_i)$. Thanks to Eq. \ref{inner product preserving}, $\{\theta(e_i)\}_{i=1}^{d}$ is a family of isometries with mutually orthogonal ranges. It suffices to prove that a total subset of $H_{[ax]}$ is contained in $\displaystyle \bigoplus_{i=1}^{d}\theta(e_i)H_{[x]}$. 

Note that $\{[uv \otimes \xi]: u \in E(a), v \in E(x), \xi \in H\}$ is total in $H_{[ax]}$. Let $u \in E(a)$, $v \in E(x)$ and $\xi \in H$ be given. Let $\{v_j\}_{j =1}^{m}$ be an orthonormal basis for $E(x)$. Then, 
\begin{align*}
[uv \otimes \xi]&=\sum_{i,j}\langle uv|e_iv_j\rangle [e_iv_j \otimes \xi]\\
&=\sum_{i,j}\langle uv|e_iv_j \rangle \theta(e_i)[v_j \otimes \xi] \\ & \in \bigoplus_{i=1}^{d}\theta(e_i)H_{[x]}.\end{align*}
Hence the proof. \hfill $\Box$

Let $d$ be the dimension function, i.e. $d(x)=\dim E(x)$ for $x \in P$. We next make $\Big\{H_{[x]}\Big\}_{[x] \in G/K}$ a measurable  field of separable Hilbert spaces.  Pick a Borel cross section $\kappa:G/ K \to G$ such that the image of $\kappa$ is contained in $P$. Let $\{e_1,e_2,\cdots \}$ be a countable family of measurable sections such that $\{e_i(x)\}_{i=1}^{d(x)}$ is an orthornormal basis for $E(x)$ for each $x \in P$.
Let $\{\xi_1,\xi_2,\cdots\}$ be an orthonormal basis for $H$. Call a section $s:G/K \to \displaystyle \coprod_{[x] \in G/K}H_{[x]}$ measurable if the map 
\[
G/K \ni [x] \to \Big \langle s([x])\Big|[e_i(\kappa([x]))\otimes \xi_j]\Big  \rangle \in \bbc\]
is measurable for every $i,j$. Denote the set of measurable sections by $\Gamma$. Then, it is not difficult to see that $\Big(\{H_{[x]}\}_{[x] \in G/K},\Gamma\Big)$ is a measurable field of Hilbert spaces over $G/K$.

 For $s \in G$, define $\tau_s:G/K \to G/K$  by $\tau_s([x])=[sx]$.  Let $\mu$ be a quasi-invariant measure on $G/K$. We can choose the measure $\mu$ such that there exists a continuous function $\omega:G \times G/K \to (0,\infty)$ and for $s \in G$,  
\[
\frac{d(\mu \circ \tau_s)}{d\mu}([x])=\omega(s,[x]).\]
We refer the reader to \cite{Folland} for the existence of such a quasi-invariant measure. 
Let 
\[
L:=\Big \{f \in \Gamma:  \int_{G/K} ||f([x])||^{2}d\mu([x])<\infty\Big \}.\]
Define an inner product on $L$ by setting 
\[
\langle f|g \rangle:=\int_{G/K} \langle f([x])|g([x]) \rangle d\mu([x]).\]
Then, $L$ is a separable Hilbert space.  

Let $a \in P$ and let $u \in E(a)$ be given. For $f \in L$, define 
\[
\phi(u)f([x]):=\omega(a^{-1},[x])^{\frac{1}{2}}\theta([a^{-1}x],u)f([a^{-1}x]).\]
It follows from Eq. \ref{inner product preserving} that $\phi(u)f \in L$. Again, by Eq. \ref{inner product preserving}, for $a \in P$, $u,v \in E(a)$ and $f \in L$, 
\begin{equation}
\label{inner product2}
\langle \phi(u)f|\phi(v)f \rangle=\langle u|v \rangle \langle f|f \rangle.\end{equation}
The above equation implies that $\phi(u)$ is a a bounded operator.  Eq. \ref{inner product2} and Eq. \ref{associativity} implies that $\phi$ is a representation of $E$ on $L$. 
We call $\phi$ the induced representation of $E$ associated with $\sigma$.

\begin{ppsn}
\label{induced0}
The representation $\phi$ is measurable and is essential. 
\end{ppsn}
\textit{Proof.} The verification that $\phi$ is measurable is left to the reader. The fact that $\phi$ is essential follows from Prop. \ref{essential}. \hfill $\Box$

We have now proved the following theorem. 
\begin{thm}
\label{induced}
Let $G$ be a locally compact, second countable, Hausdorff group, and let $P \subset G$ be a  Borel subsemigroup with non-empty interior such that $PP^{-1}=G$. Let $Q \subset P$ be a 
cofinal Borel subsemigroup such that $QQ^{-1}$ is a closed subgroup of $G$. Let $E$ be a product system over $P$. If $E|_Q$ is concrete, then $E$ is concrete. 

\end{thm}

\begin{rmrk}
Suppose $P=G$ and $Q=K$ is a closed subgroup. Let $E$ be the trivial product system over $G$, i.e. $E(x)=\bbc$ for every $x \in G$, and the multiplication is the usual multiplication
of complex numbers. Then, measurable essential representations of $E$ correspond to measurable, and hence strongly continuous  unitary representations of $G$. Similary,  measurable essential representations
of $E|_Q$ correspond to strongly continuous unitary representations of $K$. If $\sigma$ is a strongly continuous unitary representation of $K$, then the representation $\phi$ constructed 
above is  the usual induced representation of $G$ associated with $\sigma$. 
\end{rmrk}

\section{An ultraproduct construction}
In this section, we prove Thm. \ref{inductive limit}. We start with a preparatory result. 

\begin{lmma}
\label{reduction}
Let $P$ be a discrete, countable semigroup. Suppose that $E:=\{E(a)\}_{a \in P}$ is a product system of separable Hilbert spaces. If $E$ has an essential representation, then  it has an essential representation on a separable Hilbert space. 
\end{lmma}
 \textit{Proof.} Let $\phi: E \to B(K)$ be a representation of $E$ on a Hilbert space $K$ which may not be separable. Set $A:=C^{*}\{\phi(E(a)):a \in P\}$. Note that $A$ is a separable $C^{*}$-algebra. Consequently, there exists a 
separable, closed subspace $H \subset K$ such that $H$ and $H^{\perp}$ are invariant under $A$. Then, the representation $\phi$ leaves $H$ and $H^{\perp}$ invariant. The fact that $\phi$ is an essential representation of $E$ on $K$ together with the fact that $\phi$ leaves $H$ and $H^{\perp}$ invariant imply that its restriction to $H$ is again an essential representation.  \hfill $\Box$

\begin{rmrk}
\label{the most crucial remark}
Let $P$ be a measurable semigroup. Assume that $P$ is left reversible, i.e. $aP \cap bP \neq \emptyset$ for every $a,b \in P$. 
Let $E:=\{E(a)\}_{a \in P}$ be a product system over $P$. Suppose $\phi:E \to B(K)$ is a representation. Observe that for $a, b \in P$, $\phi(E(b))K \subset \phi(E(a))K$ if $b \geq a$ (Recall that we say $b \geq a$ if $b=ac$ for some $c \in P$).
Set \[
K_\infty:=\bigcap_{ a \in P}\phi(E(a))K.\]
Then, $K_\infty$ is reducing for the set $\{\phi(u): u \in E(a), a \in P\}$ of operators. 
 As in Prop. 3.2.4 of \cite{Arveson}, we can prove that $\phi$ restricted to $K_\infty$ is essential.  We call $K_\infty$ the essential part of $\phi$ and $K_\infty^{\perp}$ the singular part. Thus, to show $E$ has an essential representation on a Hilbert space, it suffices to construct a representation on a Hilbert space $K$ such that $K_\infty \neq 0$. 
  Note that since $\{\phi(E(a))K\}_{a \in P}$ is a decreasing family of subspaces  \[\displaystyle K_\infty=\bigcap_{n \geq 1}\phi(E(s_n))K\] for any cofinal sequence $(s_n)$.

\end{rmrk}

\textit{Proof of Thm. \ref{inductive limit}:} Let $P$, $P_n$ and $E$ be as in the statement of Thm. \ref{inductive limit}. Let $(s_n)_{n \geq 1}$ be a cofinal sequence in $P$ which is increasing. After renumbering, we can assume that 
for every $n \in \bbn$, $s_n \in P_n$ and there exists $t_{n+1} \in P_{n+1}$ such that $s_{n+1}=s_nt_{n+1}$. 
Set $t_1:=s_1$. 

Let $E_n$ be the product system $E$ restricted to $P_n$. Let $\phi_n$ an essential representation of $E_n$ on a separable Hilbert space $H_n$.  Let $K_0 \subset \prod_{n=1}^{\infty}H_n$ be the set of sequences $(\xi_n)_{n \geq 1}$ such that $(||\xi_n||)_{n \geq 1}$ is bounded. Fix a state $\omega$ on $\ell^{\infty}(\bbn)$ that vanishes on $c_0(\bbn)$. 
Define a semi-definite inner product $\langle~|~\rangle_{\infty}$ on $K_0$ by setting
\[
\langle \xi|\eta \rangle_{\infty}:=\omega\Big((\langle \xi_n|\eta_n \rangle)_{n=1}^{\infty}\Big)\]
for $\xi,\eta \in K_0$. We mod out the null vectors and complete to obtain a Hilbert space which we  denote by $K$. Since $\omega$ vanishes on $c_0(\bbn)$, it follows that  for $\xi:=(\xi_n)_{n \geq 1}$ and $\eta:=(\eta_n)_{n \geq 1}$ in $K_0$, the equality
\begin{equation}
\label{eventuality}
[\xi]=[\eta]   \end{equation} holds in $K$ if $\xi_n=\eta_n$ eventually. 

Let $a \in P$ and let $u \in E(a)$ be given. Choose $N \geq 1$ such that $a \in P_n$ for $n \geq N$. Define $\phi(u):K \to K$ by setting
\begin{equation}
\label{definition of phi}
\phi(u)[\xi]=[\eta]
\end{equation}
where $\eta$ is any sequence such that $\eta_n=\phi_n(u)\xi_n$ eventually. Note that $\phi_n(u)$ makes sense when $n \geq N$. 
Clearly, $\phi$ is a representation of $E$ on $K$. 

We claim that the essential part $K_\infty$ is non-zero. 
  Note that for $n \in \bbn$, $s_n=t_1t_2\cdots t_n$.  For every $k$, choose a unit vector $u_k \in E(t_k)$. For $n \geq 1$, let 
\[
v_n:=u_1u_2\cdots u_n.\]
For $n \geq 1$, let $\xi_n \in H_n$ be a unit vector, and set $\eta_n:=\phi_n(v_n)\xi_n$. Then, $[\eta]$ is a unit vector in $K$. 

Let $n \geq 1$ be given. For $k \geq n+1$, let $w_k:=u_{n+1}u_{n+2}\cdots u_k$.  Let $\gamma:=(\gamma_i)_i$ be any sequence in $K_0$ such that 
$\gamma_k=\phi_k(w_k)\xi_k$ for $k \geq n+1$. Then, for $k \geq n+1$, $\eta_k=\phi_k(v_n)\gamma_k$. By Eq. \ref{definition of phi}, we have \[[\eta]=\phi(v_n)[\gamma].\] Thus, $[\eta] \in \phi(E(s_n))K$ for every $n \geq 1$. 
Consequently, $[\eta]  \in \displaystyle \bigcap_{n \geq 1}\phi(E(s_n))K$. Thanks to Remark \ref{the most crucial remark}, $[\eta] \in K_\infty$. Thus, the essential part of $\phi$ is non-zero. The conclusion follows from Lemma \ref{reduction}
and Remark \ref{the most crucial remark}. \hfill $\Box$

\section{Proof of the main theorem}

In this section, we prove Thm. \ref{main}. 
The next four  results  are certainly well known. However,  the presentation given here can be considered different. 

\begin{ppsn}
\label{trivial case}
Every product system over the trivial group $\{0\}$ is concrete. 
\end{ppsn}
\textit{Proof.} Let $E$ be a product system over the trivial group $\{0\}$. This means that $E$ is a Hilbert space which is also an algebra over $\bbc$ such that the multiplication map 
 $E \otimes E \ni x \otimes y \to xy \in E$ is a unitary operator.  For $x \in E$, let $\phi(x):E \to E$ be the operator defined by $\phi(x)y=xy$. Then, $\phi$ defines an essential representation of $E$ on $E$. The proof is complete.  \hfill $\Box$

\begin{ppsn}
\label{natural numbers}
Let $E=\{E(n)\}_{n \in \bbn}$ be a product system over $\bbn$. Then, $E$ is concrete. Every product system over $\bbn_0$ is concrete. 
\end{ppsn}
\textit{Proof.} Let $H$ be an infinite dimensional separable Hilbert space. Fix a unitary operator $U:E(1) \otimes H \to H$. For $\xi \in E(1)$, define 
$T_u:H \to H$ by setting 
\[
T_u(\xi)=U(u \otimes \xi).\] Let $\{e_i\}_{i \in I}$ be an orthonormal basis for $E(1)$. Define a unital, normal $^*$-endomorphism of $B(H)$ by setting
\[
\alpha(A):=\sum_{i \in I}T_{e_i}AT_{e_i}^{*}.\]
Let $F:=\{F(n)\}_{n \in \bbn}$ be the product system associated with the $E_0$-semigroup $\{\alpha^n\}_{n \in \bbn}$. 
Then, $F(1)=\{T_u: u \in E(1)\}$. 
Thus, $F(1)$ and $E(1)$ have the same dimension.  It is clear that two product systems over $\bbn$ are isomorphic if and only if the corresponding fibres at $1$ have the same dimension. 
Thus, $E$ and $F$ are isomorphic. In other words, $E$ is a concrete product system. Note that $\bbn$ is a cofinal subsemigroup of $\bbn_0$. The last conclusion is  immediate from Thm. \ref{induced}. \hfill $\Box$

The following is a corollary to Prop. \ref{trivial case} and Thm. \ref{induced}.  
\begin{ppsn}
\label{group case}
Let $G$ be a locally compact group, Hausdorff, second countable group. Suppose $E=\{E(x)\}_{x \in G}$ is a product system. Then, $E$ is concrete. 
\end{ppsn}
\textit{Proof.}  In Thm. \ref{induced}, let $P=G$, and let $Q=\{e\}$. \hfill $\Box$

\begin{ppsn}
\label{singly generated}
Let $P$ be a subsemigroup of a discrete group $G$. Suppose that $P$  is singly generated, i.e. there exists $a \in P$ such that $P=\{a^n:n \in \bbn\}$. Then, every product system over $P$ is concrete. 
\end{ppsn}
\textit{Proof.} Note that either $P$ is isomorphic to $\bbn$ or $P$ is a finite group.  The conclusion is immediate from Prop. \ref{natural numbers} and Prop. \ref{group case}. \hfill $\Box$

\begin{dfn}
Let $P$ be a semigroup, and let $a \in P$. We say that $a$ is an order unit of $P$ if given $x \in P$, there exists $n \in \bbn$ such that $x \leq a^n$.
We say $P$ has an order unit if there exists $a \in P$ such that $a$ is an order unit of $P$
\end{dfn}

\begin{ppsn}
\label{order unit}
Let $G$ be a locally compact, second countable, Hausdorff topological group, and let $P \subset G$ be a Borel subsemigroup with non-empty interior such that $PP^{-1}$ is a subgroup. Suppose $P$ has an order unit.  Then, every product system over $P$ is concrete.  
\end{ppsn}
\textit{Proof.} Since $P$ has non-empty interior, $PP^{-1}=Int(P)Int(P)^{-1}$. Thus, $PP^{-1}$ is an open subgroup of $G$. Without loss of generality, we can assume that $PP^{-1}=G$. Let $a \in P$ be an order unit, i.e $\bigcup_{n=1}^{\infty}Pa^{-n}=G$. We can assume that $a^{-1} \notin P$. Otherwise, $P=G$, and we are done by Prop. \ref{group case}.   Let $Q:=\{a^n: n \geq 1\}$.  In view of Thm. \ref{induced}, it suffices to prove that the cyclic group generated by $a$, denoted $\langle a \rangle$, is closed in $G$. 

Let $(m_k)$ be a sequence of integers such that $(a^{m_k}) \to x \in G$. Since $a$ is an order unit, $x \in Int(P)a^{-r}$ for some $r>0$. Then, eventually $a^{m_k+r} \in Int(P)$. Since $Int(P)$ is a semigroup and $a \in Int(P)$,  if $m_k+r < 0$ for some $k$, then  $a^{-1} \in Int(P)$ which is a contradiction. Thus, $m_k \geq -r$ eventually. In other words, $(m_k)_k$ is bounded below. Applying the same argument to the convergent sequence $(a^{-m_k})_k$, we see that $(-m_k)_k$ is bounded below, i.e. $(m_k)$ is bounded above. Thus, there exists a subsequence $(m_{k_i})_i$ such that $(m_{k_i})_i$ is a constant sequence. Hence, $x \in \langle a \rangle$. The proof is complete. \hfill $\Box$. 

\begin{xmpl}
Two examples that satisfy the hypothesis of Prop. \ref{order unit} are given below. 

\begin{enumerate}
\item[(1)] Let $G:=H_{2n+1}$ be the Heisenberg group, and let $P:=H_{2n+1}^{+}$ be the subsemigroup of $G$ with all the matrix entries non-negative.  
\item[(2)] Let $G:=\Big\{\begin{bmatrix}
          a & b \\
          0 & 1
          \end{bmatrix}: x>0, y \in \bbr\Big\}$ be the $ax+b$-semigroup, and let \[
          P:=\Big\{\begin{bmatrix}
                        a & b \\
                        0 & 1 \end{bmatrix} \in G: a \geq 1, b \geq 0\Big\}.\]
\end{enumerate}
However, if we consider the  subsemigroup $Q$ defined by
\[
Q:=\Big \{\begin{bmatrix}
a & b \\
0 & 1 \end{bmatrix} \in G: a \leq 1, b \geq 0 \Big\},\]
then $Q$ is right Ore but $Q$ does not have an order unit. 
\end{xmpl}

In the connected case, we record the observation that normal semigroups have order units. 
\begin{ppsn}
\label{connected}
Let $G$ be a locally compact, second countable, Hausdorff topological group. Suppose that $G$ is connected. Let $P \subset G$ be a Borel subsemigroup with non-empty interior. If $P$ is normal in $G$,  then every product system over $P$ is concrete. 
\end{ppsn}
\textit{Proof.} Since $P$ is normal, $PP^{-1}$ is a group. In view of Prop. \ref{order unit}, it suffices to prove that $P$ has an order unit $a \in P$.  Let $a \in Int(P)$. As $P$ is normal,  $Int(P)$ is normal. The normality of $Int(P)$ implies that $\bigcup_{n=1}^{\infty}Int(P) a^{-n}$ is a  subgroup. Moreover,  $\bigcup_{n=1}^{\infty}Int(P) a^{-n}$ is open. Since $G$ is connected, 
\[
G=\bigcup_{n=1}^{\infty}Int(P) a^{-n}.\]
Hence, $a$ is an order unit for $P$. \hfill $\Box$

\begin{ppsn}
\label{exhausting normal semigroups}
Let $G$ be a discrete countable group, and let $P \subset G$ be a semigroup that is normal in $G$, i.e. $gPg^{-1}=P$ for every $g \in G$. Then, $P$ is right Ore. Also, there exists a sequence of subsemigroups $(P_n)_n$ such that 
\begin{enumerate}
\item[(1)] for every $n \in \bbn$, $P_n$ is right Ore, 
\item[(2)] for every $n \in \bbn$, $P_n$ has an order unit, and
\item[(3)] the sequence $(P_n)_n \nearrow P$. 
\end{enumerate}
\end{ppsn}
\textit{Proof.} The fact that $PP^{-1}$ is a subgroup follows immediately from the normality assumption on $P$. 
Let $(s_n)_n$ be an increasing cofinal sequence in $P$. For $n \geq 1$, let 
\[
P_n:=\{x \in P: \textrm{ there exists $k \in \bbn$ such that $x \leq s_n^k$}\}.\]
The normality of $P$ ensures that $P_n$ is a semigroup, and $G_n=P_nP_n^{-1}$ is a subgroup of $G$. Moreover, $s_n \in P_n$ is an order unit of $P_n$.   
Thanks to the normality of $P$, $P_n \subset P_{n+1}$. The fact that $(s_n)_n$ is cofinal implies that $\bigcup_{n \geq 1}P_n=P$. \hfill $\Box$.

The following result is a corollary to Prop. \ref{inductive limit}, Prop. \ref{exhausting normal semigroups} and to Prop. \ref{order unit}. 
\begin{crlre}
\label{discrete}
Let $G$ be a discrete countable group, and let $P \subset G$ be a semigroup that is normal in $G$. Then, every product system over $P$ is concrete. 
\end{crlre}

We have now proved  Thm. \ref{main}   except for the locally compact abelian case which we prove next. 

\begin{thm}
\label{main_abelian}
Let $G$ be a second countable, locally compact abelian group, and let $P \subset G$ be a Borel subsemigroup with non-empty interior. Then, every product system over $P$ is concrete. 
\end{thm}
\textit{Proof.} Since $P$ has non-empty interior, we do not lose generality  if  we assume that $P-P=Int(P)-Int(P)=G$.  Hence, $Int(P)$ is a cofinal semigroup. 
 In view of Thm. \ref{induced},  it suffices to consider the case when $P$ is open. Thus, let $P \subset G$ be an open subsemigroup such that $P-P=G$. 

Let $(s_n)_n$ be a cofinal sequence in $P$. We can assume that $s_n \leq s_{n+1}$. Note that, by passing to a subsequence, we can assume that, for every $n$, $s_n \notin -P$. Otherwise, the cofinality of $(s_n)_n$ implies that $P=G$ in which case we are done by Prop. \ref{group case}. 
For $n \in \bbn$, let 
\[
K_n:=\{x \in G: \textrm{there exists $k \in \bbn$ such that $-ks_n \leq x \leq ks_n$}\}.\]
Since $P$ is open,  $K_n$ is an open subgroup (and hence closed). Note that $K_n \subset K_{n+1}$, and $\bigcup_{n=1}^{\infty}K_n=G$. 

\

\textit{Case (i):} Suppose that the sequence $(K_n)_n$ stabilises. Choose $N$ such that $K_N=G$. In this case, $s_N$ is an order unit of $P$, and we can apply Prop. \ref{order unit} to conclude the result. 
 
 \textit{Case (ii):} Assume that the sequence $(K_n)_n$ does not stabilise.   Then, by passing to a subsequence, we can assume  for every $n$, $K_n \neq K_{n+1}$. This, in particular, implies that $rs_{n+1} \notin K_n$ unless $r=0$, i.e. $s_n+K_{n-1}$ is not of finite order in the discrete group $K_n/K_{n-1}$. Let $Q$ be the subsemigroup generated by $\{s_n: n \geq 1\}$. Clearly, $Q$ is a cofinal semigroup.  Let $H:=Q-Q$. We claim that $H$ is closed in $G$. Since $\{K_n\}_{n \geq 1}$ is an increasing open cover, it suffices to prove that $H \cap K_n$ is closed for each $n$. Since for $r \neq 0$, $rs_k \notin K_n$  for $k >n$, 
\[
H \cap K_n=\{\sum_{k=1}^{n}m_ks_k: m_k \in \bbz\}.\]
Let $n \geq 1$ be given. Let $(x_r)_r$ be a sequence in $H \cap K_n$ such that $(x_r)_r \to x$. Since $K_n$ is closed, $x \in K_n$. Write, for $r \geq 1$, 
\[
x_r:=\sum_{i=1}^{n}m_i^{(r)}s_i\]
where $m_i \in \bbz$. Since $K_{n-1}$ is a clopen subgroup of $K_n$, $K_n/K_{n-1}$ is discrete. This forces that the sequence $(m_n^{(r)}s_n+K_{n-1})_r$ is eventually constant. By assumption, $s_n+K_{n-1}$ is not of finite order. Hence, the sequence $(m_{n}^{(r)})_r$ is eventually constant. 

Then, the sequence $(\sum_{i=1}^{n-1}m_i^{(r)}s_i)_r$ converges in $K_{n-1}$ (note that $K_{n-1}$ is closed in $K_n$). Proceeding recursively, we can conclude that for every $i \in \{1,2,\cdots,n\}$, the sequence $(m_i^{(r)})_r$ is eventually constant. Hence, $H \cap K_n$ is closed for every $n$. Consequently, $H$ is closed. 
Thus, $Q$ is a countable, cofinal semigroup such that $Q-Q$ is closed. Appealing to Thm. \ref{induced} and Corollary \ref{discrete}, we can conlcude that every product system over $P$ is concrete. \hfill $\Box$

\section{A non-example}
In this section, we show that if $P$ is a discrete semigroup that does not embedd inside a group, then there is a product system over $P$ that is not concrete. 
 Recall that a semigroup $P$ is said to be group-embeddable if there exists a group $G$ and a semigroup homomorphism $\tau:P \to G$ such that $\tau$ is injective. As alluded to in the introductory section, there are countable, cancellative semigroups that do not embedd in a group, and the first such example was due to Malcev (\cite{Malcev}).  We refer the reader to \cite{Edwardes} for a recent construction of such examples of semigroups that are motivated by questions in $C^{*}$-algebras. For a historical account of the embedding problem of monoids inside groups, we refer the reader to \cite{Hollings}. 
 
 Let $P$ be a discrete, cancellative, countable semigroup that is not group-embeddable. Let $V$ be the left regular representation of $P$ on $\ell^2(P)$, i.e. for $x \in P$, let $V_x$ be the isometry on $\ell^2(P)$ defined by 
 \[
 V_{x}(\delta_y)=\delta_{xy}.\]
Here, $\{\delta_y:y \in P\}$ is the standard orthonormal basis of $\ell^2(P)$. Then, $V:=\{V_x\}_{x \in P}$ is a semigroup of isometries on $\ell^2(P)$. Let $\alpha:=\{\alpha_x\}_{x \in P}$ be the CCR flow associated with $V$ on $B(\Gamma(\ell^2(P))$. Here, $\Gamma(\ell^2(P))$ denotes the symmetric Fock space of $\ell^2(P)$. Recall that, for $x \in P$, the endomorphism $\alpha_x$ on the set of Weyl operators $\{W(\xi):\xi \in \ell^2(P)\}$ is given by the equation
\[
\alpha_x(W(\xi))=W(V_x\xi).\]
Observe that the map $x \to \alpha_x$ is injective, i.e. for $x, y \in P$, $\alpha_x=\alpha_y$ if and only if $x=y$. 

Let $E:=\{E(x)\}_{x \in P}$ be the product system associated with $\alpha$. Then, the opposite product system $E^{op}$ is a product system over the opposite semigroup $P^{op}$. The opposite $E^{op}$ is defined as follows:  for $x \in P^{op}$, the fibre $E^{op}(x)=E(x)$ and the product is defined by 
\[
u \odot v=vu\]
for $u \in E^{op}(x)$ and $v \in E^{op}(y)$.  Then, $E^{op}$ is a product system over $P^{op}$. 

\begin{ppsn}
\label{example}
With the above notation, the product system $E^{op}$ over $P^{op}$ is not concrete. 
\end{ppsn}
\textit{Proof.} Suppose $\phi$ is an essential representation of $E^{op}$ on a separable Hilbert space $H$. In other words,  the map $\displaystyle \phi:\coprod_{x \in P}E(x) \to B(H)$ satisfies the following conditions:
\begin{enumerate}
\item[(i)] for $x \in P$, $u,v \in E(x)$, $\phi(v)^*\phi(u)=\langle u |v \rangle$,
\item[(ii)] for $x, y \in P$, $u \in E(x)$ and $v \in E(y)$, $\phi(uv)=\phi(v)\phi(u)$, and
\item[(iii)] for $x \in P$, $[\phi(E(x))H]=H$.  
\end{enumerate} 

We follow Arveson (\cite{Arv_Spectral1990}) to make use of  $\phi$ and the concrete representation of $E$ on $\Gamma(\ell^2(P))$ to define   a semigroup of unitary operators on the Hilbert space $\Gamma(\ell^2(P))\otimes H$. 
As mentioned in the introduction, Arveson employed this trick to show that  every  one parameter $E_0$-semigroup on $B(H)$ extends to a group of automorphisms. 

Let $K:=\Gamma(\ell^2(P))\otimes H$, and let $\mathcal{U}(K)$ be the group of unitary operators on $K$. For $x \in P$, let $d(x)$ be the dimension of $E(x)$, and let $\{T_i(x)\}_{i=1}^{d(x)}$ be an orthonormal basis for $E(x)$. Define
\begin{equation}
\label{arveson unitaries}
U_x:=\sum_{i=1}^{d(x)}T_i(x)\otimes \phi(T_i(x))^*.\end{equation}
If we set $w_i:=T_{i}(x) \otimes \phi(T_i(x))^*$, then $\{w_i\}_{i=1}^{d(x)}$ is a family of partial isometries such that 
\[
\sum_{i}w_iw_i^{*}=\sum_{i}w_i^{*}w_i=1.\]
Hence, $U_x$ is a unitary operator. 

Note that $U_x$ does not depend on the orthonormal basis chosen. This along with the fact that $\phi$ is an anti-multiplicative and the fact that the multiplication map \[E(x) \otimes E(y) \ni S \otimes T \to ST \in E(xy)\] is a unitary operator imply that $U_xU_y=U_{xy}$ for $x, y \in P$. In other words, the map $P \ni x \to U_x \in \mathcal{U}(K)$ is a semigroup homomorphism. 

We claim that  for $A \in B(\Gamma(\ell^2(P)))$ and $x \in P$, 
\begin{equation}
\label{extension}
U_x(A \otimes 1)U_x^*=\alpha_x(A) \otimes 1.
\end{equation}
To see the last equality, let $x \in P$ and $A \in B(\ell^2(P))$, and calculate as follows to observe that 
\begin{align*}
U_x(A \otimes 1)&=\sum_{i=1}^{d(x)}T_i(x)A \otimes \phi(T_i(x))^*\\
&=\sum_{i=1}^{d(x)}\alpha_x(A)T_i(x) \otimes \phi(T_i(x))^{*} ~~~(\textrm{since $T_i(x) \in E_\alpha(x)$})\\
&=(\alpha_x(A) \otimes 1)\sum_{i=1}^{d(x)}T_i(x)\otimes \phi(T_i(x))^*\\
&=(\alpha_x(A) \otimes 1)U_x.
\end{align*} 
This proves the claim. 

Since $\alpha_x=\alpha_y \implies x=y$, it follows from Eq. \ref{extension} that the map $P \ni x \to U_x  \in \mathcal{U}(K)$ is an embedding which is a contradiction as $P$ is not group-embeddable. Hence, the product system $E^{op}$ over $P^{op}$ is not concrete. \hfill $\Box$

The ideas of the above proof can be used to derive a necessary condition a discrete cancelative semigroup must satisfy if every product system over it is concrete. Our condition says that if we expect that every product system over a semigroup $P$ comes from an $E_0$-semigroup over $P$, then the semigroup $P$ needs to be `amenable' to dilation theory at least at the level of sets. To explain it, let us fix some terminology. Let $P$ be a countable, cancellative, discrete semigroup. For a non-empty set $X$, denote the set of injective maps from $X \to X$ by $End(X)$, and denote the set of bijective maps from $X \to X$ by $Perm(X)$. Note that $End(X)$ is a monoid, and $Perm(X)$ is the permutation group. 
A map $\phi:P \to End(X)$ is called an injective right action of $P$ on $X$ if $\phi$ is anti-multiplicative, i.e. \[\phi_a \circ \phi_b=\phi_{ba}\] for $a,b \in P$. We call $\phi$ a bijective action if the range of $\phi$ is contained in $Perm(X)$. As usual, for $x \in X$ and $a \in P$, we denote $\phi_a(x)$ by $xa$.

\begin{dfn}
Let $X$ be a set. Let $\phi:P \to End(X)$ be an injective right action. We say $\phi$ admits a bijective dilation if there exists a set $Y$, a bijective right action $\psi:P \to Perm(Y)$, and an injective map $i:X \to Y$ such that for $a \in P$,
\[
\psi_a(i(x)):=i(\phi_a(x)).\]
\end{dfn}

We say a semigroup $P$ has Property \textbf{(D)} if every injective right action admits a bijective dilation.

\begin{ppsn}
Let $P$ be a discrete, countable semigroup.  If every product system over $P$ is concrete, then $P$ has Property \textbf{(D)}. 
\end{ppsn}
\textit{Proof.} Let $\phi:P \to End(X)$ be an injective right action. To show that $\phi$ has a bijective dilation, it suffices to consider the case when $X$ is countable. Define an isometric representation of $P^{op}$ on $\ell^2(X)$
by the formula 
\[
V_{a}(\delta_x)=\delta_{xa}.\]
Here, $\{\delta_x\}_{x \in X}$ is the standard orthonormal basis of $\ell^2(X)$. As in Prop. \ref{example}, consider the CCR flow associated with $V$ over $P^{op}$, denote its product system by $E$, and let $F:=E^{op}$ be the opposite product system
which is now a product system over $P$. If $F$ admits an essential representation on a Hilbert space $H$, then as in Prop. \ref{example}, we can define a family of unitary operators  $\{U_a\}_{a \in P}$ on the Hilbert space  $K:=\Gamma(L^2(X))\otimes H$
such that 
\begin{enumerate}
\item[(C1)] for $a,b \in P$, $U_aU_b=U_{ba}$, and
\item[(C2)] for $a \in P$ and $\xi \in L^2(X)$, $U_a(W(\xi) \otimes 1)U_a^{*}=W(V_a\xi)$.
\end{enumerate}
Let $Y$ be the unitary group of $K$.  For $a \in P$, let $\psi_a:Y \to Y$ be the bijection defined by  $\psi_a:=Ad(U_a)$. Condition (C1) ensures that $\psi:P \to End(Y)$ is a right action. Also, $\psi$ is a bijective action. For $x \in X$, let $i:X \to Y$ be defined by
\[
i(x)=W(\delta_x)\otimes 1.\]
Then, Condition (C2) ensures that $\psi$ is a dliation of $\phi$. The proof is complete. \hfill $\Box$

\begin{rmrk}
Let $P$ be a cancellative semigroup. 
 If $P$ is cancellative, then $P$ acts  on itself by multiplcation on the right. If this action admits a bijective dilation, then it is immediate to see that $P$ is group-embeddable. Thus, if $P$ has Property \textbf{(D)}, then $P$ is group-embeddable. 
  Left reversible semigroups and the free semigroup $\mathbb{F}_{n}^{+}$ satisfy Property \textbf{(D)}. 
    Does every subsemigroup of a group have Property \textbf{(D)}? The author does not know the answer to this. 
   Li mentions in \cite{Li_Oberwolfach} (Chapter 5, Section 5.4.1) that if $R$ is an integral domain that is not a field, then the $ax+b$-semigroup $R \rtimes R^{*}$ is right reversible, but not left reversible.  Does this semigroup have property \textbf{(D)}? 
    Does every product system over $R \times R^{*}$ come from an $E_0$-semigroup? 
     
\end{rmrk}

\section{Extension of a result of Liebscher}
 We record yet another instance where  the group of unitaries considered in Prop. \ref{example} and the automorphism group that it determines comes handy. 
We show that in the setting of  abelian subsemigroups of locally compact groups, two product systems are isomorphic if and only if they are algebraically isomorphic.  In the $1$-parameter case, this result is due to Liebscher (see Corollary 7.16 of \cite{Liebscher}). The author  believes that the proof presented in this section  will have  its own appeal. 

We start with a  lemma which is an application of the fact that symmetric multipliers on abelian groups are coboundaries. 

\begin{lmma}
\label{Wigner_kind}
Let $G$ be a locally compact, second countable, Hausdorff abelian group. Let $H$ be a separable Hilbert space. Suppose $U:=\{U_x\}_{x \in G}$ is a group of unitary operators on $H$.  For $x \in G$, let $\gamma_x:=Ad(U_x)$. Then, $\gamma:=\{\gamma_x\}_{x \in G}$ is an $E_0$-semigroup over $G$ if and only if there exists a homomorphism $\chi:G \to \bbt$, not necessarily continuous, such that $\chi U:=\{\chi(x)U_x\}_{x \in G}$ is a weakly measurable group of unitary operators. 
\end{lmma}
\textit{Proof.} The `only if' part is obvious. Assume now that $\gamma$ is an $E_0$-semigroup. Then, there exists a weakly measurable family of unitary operators $W:=\{W_x\}_{x \in G}$ on $H$ such that $\gamma_x=Ad(W_x)$ for $x \in G$ \footnote{The proof of the existence of such a weakly measurable family is not difficult and it follows from the representation theory of the $C^{*}$-algebra of compact operators. This is also a special case  of the fact that the product system of an $E_0$-semigroup does satisfy the axioms required for it be a product system. We can also apply a result from \cite{Arveson_invitation} (see Corollary to Lemma 4.1.4 of \cite{Arveson_invitation}) to $K(H)$ and  can justify the existence of such a weakly measurable family.}. 
 The equality $\gamma_x \circ \gamma_y=\gamma_{x+y}$ implies that for $x,y \in G$, there exists $\omega(x,y) \in \bbt$ such that 
\begin{equation}
\label{Borel multiplier1}
W_{x}W_y=\omega(x,y)W_{x+y}.\end{equation}
Clearly, $\omega:G \times G \to \bbt$ is Borel measurable. The fact that $(W_xW_y)W_z=W_x(W_yW_z)$ for $x,y,z \in G$ implies that $\omega$ is a multiplier, i.e for $x,y,z \in G$, 
\[
\omega(x,y)\omega(x+y,z)=\omega(x,y+z)\omega(y,z).\]

Since, for $x \in G$, $Ad(U_x)=Ad(W_x)$, it follows that $U_x^*W_x$ is a scalar for every $x \in G$. Let $f:G \to \bbt$ be the map such that $W_x=f(x)U_x$ for $x \in G$. Eq. \ref{Borel multiplier1} together with the assumption that $\{U_x\}_{x \in G}$ is a group of unitary operators imply that for $x,y \in G$, 
\[
\omega(x,y)=\frac{f(x)f(y)}{f(x+y)}.\]
Hence, $\omega$ is symmetric. It follows from  Lemma 7.2 of \cite{Kleppner} that there exists a measurable function $g:G \to \bbt$ such that for $x,y \in G$, 
\[
\frac{f(x)f(y)}{f(x+y)}=\omega(x,y)=\frac{g(x)g(y)}{g(x+y)}.\] Then, $\chi=\overline{g}f$ is a homomorphism, and $\chi U=\{\chi(x)U_x\}_{x \in G}=\{\overline{g}(x)W_x\}_{x \in G}$ is weakly measurable. The proof is complete. \hfill $\Box$

The next theorem shows that, in the commutative setting, any two measurable structures on an algebraic product system differ by a character which need not be measurable. 
\begin{thm}
\label{Liebscher}
Let $G$ be a locally compact abelian group, and let $P \subset G$ be a Borel subsemigroup with non-empty interior. 
Let $E:=\{E(x)\}_{x \in P}$ be an algebraic product system. Suppose $\Gamma_0$ and $\Gamma_1$ are subsets of the set of  sections of $E$ such that $(E,\Gamma_0)$ and $(E,\Gamma_1)$ are product systems. Then, there exists a multiplicative map $\chi:P \to \bbt$ such that $\Gamma_1=\chi \Gamma_0$. 
\end{thm}
\textit{Proof.} Without loss of generality, we can assume $P-P=G$. 
Let $\phi$ be an  essential representation of the product system $(E,\Gamma_0)$ on a separable Hilbert space $H$ such that $\phi$ is measurable.  Let $\psi$ be an essential representation of $(E^{op},\Gamma_1)$ on a separable Hilbert space $K$ such that $\psi$ is measurable. 
Thm. \ref{main_abelian} guarantees the existence of $\phi$ and $\psi$. For $x \in P$, let $d(x):=\dim E(x)$.
Let $e_1,e_2,\cdots$ be a countable family of sections in $\Gamma_0$ such that for each $x \in P$, $\{e_i(x)\}_{i=1}^{d(x)}$ is an orthonormal basis for $E(x)$. Let $f_1,f_2,\cdots$ be a countable family of sections in $\Gamma_1$ such that for each $x \in P$, $\{f_i(x)\}_{i=1}^{d(x)}$ is an orthonormal basis for $E(x)$. 

Let $\alpha:=\{\alpha_x\}_{x \in P}$ be the $E_0$-semigroup over $P$ on $B(H)$ determined by the representation $\phi$ of $(E,\Gamma_0)$, i.e. the endomorphism $\alpha_x$, for $x \in P$, is given by 
the formula 
\begin{equation}
\label{endo alpha}
\alpha_x(A)=\sum_{i=1}^{d(x)}\phi(e_i(x))A \phi(e_i(x))^{*}
\end{equation}
for $A \in B(H)$. 
Similalry, let $\beta:=\{\beta_x\}_{x \in P}$ be the $E_0$-semigroup over $P$ on $B(K)$ defined by 
\begin{equation}
\label{endo beta}
\beta_x(B)=\sum_{i=1}^{d(x)}\psi(f_i(x))B \psi(f_i(x))^{*}
\end{equation}
for $B \in B(K)$. 

 For $x \in P$, let $\{v_i(x)\}_{i=1}^{d(x)}$ be an orthonormal basis of $E(x)$, and set
\[
U_x:=\sum_{i=1}^{d(x)}\phi(v_i(x))\otimes \psi(v_i(x))^{*}.\]
The definition of $U_x$ does not depend on the choice of the  orthonormal basis. Hence, for $x,y\in P$,  \begin{equation}
\label{unitaries arveson}
U_{x+y}=U_xU_y.
\end{equation}

For $x \in G$, write $x=a-b$ with $a,b \in P$, and set $\widetilde{U}_x=U_{b}^{*}U_a$. Using Eq. \ref{unitaries arveson}, we can prove that $\widetilde{U}_x$ does not depend on the choice of $a$ and $b$. Hence, $\widetilde{U}_x$ is well defined. Again using Eq. \ref{unitaries arveson}, we can prove that $\{\widetilde{U}_x\}_{x \in G}$ satisfy the equation 
\begin{equation}
\label{unitaries arveson1}
\widetilde{U}_{x+y}=\widetilde{U}_x\widetilde{U}_y
\end{equation}
for $x,y \in G$. Also, $\widetilde{U}_x=U_x$ for $x \in P$. For $x \in G$, let $\gamma_x$ be the automorphism of $B(H \otimes K)$ defined by $\gamma_x:=Ad(\widetilde{U}_x)$. 
Note that for $x \in P$, $A \in B(H)$ and $B \in B(K)$, \begin{equation}
\label{pairing1}\gamma_x(A \otimes 1)=\alpha_x(A) \otimes 1\end{equation}
and \begin{equation}
\label{pairing2} \gamma_{-x}(1 \otimes B)=1 \otimes \beta_{x}(B).\end{equation}

\textit{Claim:} $\gamma:=\{\gamma_x\}_{x \in G}$ is an $E_0$-semigroup over $G$ on $B(H \otimes K)$. 

By definition, $\gamma$ is a group of automorphisms, and we only need to check the measurability axiom. 
 It suffices to check that for $A \in B(H)$ and $B \in B(K)$, the map
\[
G \ni x \to \gamma_x(A \otimes 1) \in B(H \otimes K)\]
and the map 
\[
G \ni x \to \gamma_x(1 \otimes B) \in B(H \otimes K)\]
are weakly measurable.

Let $(s_n)_{n=1}^{\infty}$ be a  cofinal sequence in $P$. Then, \[\bigcup_{n=1}^{\infty}(P-s_n)=G=\bigcup_{n=1}^{\infty}(-P+s_n).\] 
Let $A \in B(H)$ be given. Then, for $n \in \bbn$ and $x \in P-s_n$, 
\[
\gamma_x(A \otimes 1)=\gamma_{s_n}^{-1}\circ \gamma_{x+s_n}(A \otimes 1)=\gamma_{s_n}^{-1}(\alpha_{x+s_n}(A) \otimes 1)=\widetilde{U}_{s_n}^{*}(\alpha_{x+s_n}(A) \otimes 1)\widetilde{U}_{s_n}. \]
It follows from the above equation  that the map \[P-s_n \ni x \to \gamma_x(A \otimes 1) \in B(H \otimes K)\] is weakly measurable. Since $\{P-s_n: n \in \bbn\}$ covers $G$, it follows that the map $G \ni x \to \gamma_x(A \otimes 1) \in B(H \otimes K)$ is weakly measurable. 

Arguing in the same fashion by using  Eq. \ref{pairing2} and by noting that $\{s_n-P: n \geq 1\}$ is a measurable cover of $G$, we prove that the map 
\[
G \ni x \to \gamma_x(1 \otimes B) \in B(H \otimes K)\]
is weakly measurable for every $B \in B(K)$.  This proves the claim. 

 By Lemma \ref{Wigner_kind}, there exists a homomorphism $\chi:G \to \bbt$ such that $\chi U:=\{\chi(x)U_x\}_{x \in G}$ is weakly measurable. 
For $i \in \bbn$, let $T_i:P \to B(H \otimes K)$ and $R_i:P \to B(H \otimes K)$ be defined by \begin{align*}
T_i(x):&=\phi(e_i(x))\otimes 1,\\
R_i(x):&=1 \otimes \psi(f_i(x)).
\end{align*}
As the definition of $U_x$ does not depend on the choice of the orthonormal basis,  observe that 
 for $i,j  \in \bbn$ and $x \in P$, 
\begin{align*}
T_i(x)^{*}\chi(x)U_xR_j(x)&=\chi(x)  (\phi(e_i(x))^* \otimes 1)\Big(\sum_{k=1}^{\infty}\phi(e_k(x) \otimes \psi(e_k(x))^*\Big)(1 \otimes \psi(f_j(x)))\\
&=\chi(x) \langle e_i(x)|f_j(x) \rangle.
\end{align*}
Since $T_i,R_j$ and $\chi U$ are weakly measurable, the above equality implies the map 
\[
P \ni x \to \langle \chi(x)e_i(x)|f_j(x)\rangle =\langle e_i(x)|\overline{\chi(x)}f_j(x) \rangle \in \bbc\]
is measurable for every  $i,j$. Hence, $\chi e_i \in \Gamma_1$ for every $i$ and $\overline{\chi}f_j \in \Gamma_0$ for every $j$. Hence, $\Gamma_1=\chi \Gamma_0$.  This completes the proof. \hfill $\Box$

\begin{crlre}
Let $G$ be a locally compact abelian group, and let $P \subset G$ be a Borel subsemigroup with non-empty interior. 
For $i=1,2$, let $E_i$ be a product system over $P$ with $\Gamma_i$ being the space of measurable sections of $E_i$. Then, $E_1$ and $E_2$ are isomorphic as product systems if and only if they are isomorphic as algebraic product systems. 
\end{crlre}
\textit{Proof.} Suppose $E_1$ and $E_2$ are isomorphic as algebraic product systems. This means that there exists a field of unitary operators $\theta:=\{\theta_x\}_{x \in P}$ such that \[
\theta_{x+y}(uv)=\theta_x(u)\theta_y(v)\]
for $x,y \in P$, $u \in E_1(x)$ and $v \in E_2(x)$. 
Let $\widetilde{\Gamma}_2=\theta(\Gamma_1)$. Then, by Thm. \ref{Liebscher}, there exists a homomorphism $\chi:G \to \bbt$ such that $\widetilde{\Gamma}_2=\chi \Gamma_2$. For $x \in P$, let $\widetilde{\theta}_x=\overline{\chi(x)}\theta_x$. Then, $\widetilde{\theta}:=\{\widetilde{\theta}_x\}_{x \in P}$ is a measurable field of unitary operators such that for $x, y\in P$, $u \in E_1(x)$ and $v \in E_2(x)$, 
\[
\widetilde{\theta}_{x+y}(uv)=\widetilde{\theta}_x(u)\widetilde{\theta}_y(v).\]
Hence, $E_1$ and $E_2$ are isomorphic. The proof is complete. \hfill $\Box$

\bibliography{references}

\def\cprime{$'$} \def\cprime{$'$} \def\cprime{$'$}
\providecommand{\bysame}{\leavevmode\hbox to3em{\hrulefill}\thinspace}
\providecommand{\MR}{\relax\ifhmode\unskip\space\fi MR }
\providecommand{\MRhref}[2]{%
  \href{http://www.ams.org/mathscinet-getitem?mr=#1}{#2}
}
\providecommand{\href}[2]{#2}
\begin{thebibliography}{10}

\bibitem{Meyer}
Suliman Albandik and Ralf Meyer, \emph{Product systems over {O}re monoids},
  Doc. Math. \textbf{20} (2015), 1331--1402.

\bibitem{Alexi}
Alexis Alevras, \emph{One parameter semigroups of endomorphisms of factors of
  type {$\rm II_1$}}, J. Operator Theory \textbf{51} (2004), no.~1, 161--179.

\bibitem{Arveson_invitation}
William Arveson, \emph{An invitation to {$C\sp*$}-algebras}, Springer-Verlag,
  New York-Heidelberg, 1976, Graduate Texts in Mathematics, No. 39.

\bibitem{Arv_Fock}
\bysame, \emph{Continuous analogues of {F}ock space {I}}, Mem. Amer. Math. Soc.
  \textbf{80} (1989), no.~409.

\bibitem{Arv_Fock4}
\bysame, \emph{Continuous analogues of {F}ock space. {IV}. {E}ssential states},
  Acta Math. \textbf{164} (1990), no.~3-4, 265--300.

\bibitem{Arv_Spectral1990}
\bysame, \emph{The spectral {$C^*$}-algebra of an {$E_0$}-semigroup}, Operator
  theory: operator algebras and applications, {P}art 1 ({D}urham, {NH}, 1988),
  Proc. Sympos. Pure Math., vol. 51, Part 1, Amer. Math. Soc., Providence, RI,
  1990, pp.~1--15.

\bibitem{Arveson}
\bysame, \emph{Noncommutative dynamics and {$E$}-semigroups}, Springer
  Monographs in Mathematics, Springer-Verlag, New York, 2003.

\bibitem{Arv}
\bysame, \emph{On the existence of {$E_0$}-semigroups}, Infin. Dimens. Anal.
  Quantum Probab. Relat. Top. \textbf{9} (2006), no.~2, 315--320.

\bibitem{Li_Oberwolfach}
Joachim Cuntz, Siegfried Echterhoff, Xin Li, and Guoliang Yu,
  \emph{{$K$}-theory for group {$C^*$}-algebras and semigroup
  {$C^*$}-algebras}, Oberwolfach Seminars, vol.~47, Birkh\"{a}user/Springer,
  Cham, 2017.

\bibitem{Sims_Kumjian}
Valentin Deaconu, Alex Kumjian, David Pask, and Aidan Sims, \emph{Graphs of
  {$C^\ast$}-correspondences and {F}ell bundles}, Indiana Univ. Math. J.
  \textbf{59} (2010), no.~5, 1687--1735.

\bibitem{Dinh}
Hung~T. Dinh, \emph{Discrete product systems and their {$C^*$}-algebras}, J.
  Funct. Anal. \textbf{102} (1991), no.~1, 1--34.

\bibitem{Dinh_Ktheory}
\bysame, \emph{On generalized {C}untz {$C^*$}-algebras}, J. Operator Theory
  \textbf{30} (1993), no.~1, 123--135.

\bibitem{Edwardes}
Milo Edwardes and Daniel. Heath, \emph{A collection of cancellative right lcm,
  not group-embeddable monoids}, arxiv:2405.20197v1.

\bibitem{Folland}
Gerald~B. Folland, \emph{A {C}ourse in {A}bstract {H}armonic {A}nalysis},
  Studies in Advanced Mathematics, CRC Press, Boca Raton, FL, 1995.

\bibitem{Fowler_2002}
Neal~J. Fowler, \emph{Discrete product systems of {H}ilbert bimodules}, Pacific
  J. Math. \textbf{204} (2002), no.~2, 335--375.

\bibitem{Hollings}
Christopher Hollings, \emph{Embedding semigroups in groups: not as simple as it
  might seem}, Arch. Hist. Exact Sci. \textbf{68} (2014), no.~5, 641--692.

\bibitem{Kleppner}
Adam Kleppner, \emph{Multipliers on abelian groups}, Math. Ann. \textbf{158}
  (1965), 11--34.

\bibitem{Kumjian_higher}
Alex Kumjian and David Pask, \emph{Higher rank graph {$C^\ast$}-algebras}, New
  York J. Math. \textbf{6} (2000), 1--20.

\bibitem{Szymanski}
Bartosz~Kosma Kwa\'sniewski and Wojciech Szyma\'nski, \emph{Topological
  aperiodicity for product systems over semigroups of {O}re type}, J. Funct.
  Anal. \textbf{270} (2016), no.~9, 3453--3504.

\bibitem{Liebscher}
Volkmar Liebscher, \emph{Random sets and invariants for (type {II}) continuous
  tensor product systems of {H}ilbert spaces}, Mem. Amer. Math. Soc.
  \textbf{199} (2009), no.~930.

\bibitem{Malcev}
A.~Malcev, \emph{On the immersion of an algebraic ring into a field}, Math.
  Ann. \textbf{113} (1937), no.~1, 686--691.

\bibitem{Srinivasan_Oliver1}
Oliver~T. Margetts and R.~Srinivasan, \emph{Invariants for {${E}_0$}-semigroups
  on {${II}_1$} factors}, Comm. Math. Phys. \textbf{323} (2013), no.~3,
  1155--1184.

\bibitem{Murugan_sundar_continuous}
S.~P. Murugan and S.~Sundar, \emph{On the existence of {$E_0$}-semigroups---the
  multiparameter case}, Infin. Dimens. Anal. Quantum Probab. Relat. Top.
  \textbf{21} (2018), no.~2, 1850007, 20.

\bibitem{Murugan_sundar_discrete_mathsci}
\bysame, \emph{An essential representation for a product system over a finitely
  generated subsemigroup of {$\Bbb Z^d$}}, Proc. Indian Acad. Sci. Math. Sci.
  \textbf{129} (2019), no.~2.

\bibitem{Murugan_Sundar}
S.P. Murugan and S.~Sundar, \emph{${E}_{0}^{P}$-semigroups and {P}roduct
  systems}, arxiv/math.OA:1706.03928.

\bibitem{Pimsner}
Michael~V. Pimsner, \emph{A class of {$C^*$}-algebras generalizing both
  {C}untz-{K}rieger algebras and crossed products by {${\bf Z}$}}, Free
  probability theory ({W}aterloo, {ON}, 1995), Fields Inst. Commun., vol.~12,
  Amer. Math. Soc., Providence, RI, 1997, pp.~189--212.

\bibitem{Renault_Pgraph}
Jean~N. Renault and Dana~P. Williams, \emph{Amenability of groupoids arising
  from partial semigroup actions and topological higher rank graphs}, Trans.
  Amer. Math. Soc. \textbf{369} (2017), no.~4, 2255--2283.

\bibitem{Rennie_Sims}
Adam Rennie, David Robertson, and Aidan Sims, \emph{Groupoid {F}ell bundles for
  product systems over quasi-lattice ordered groups}, Math. Proc. Cambridge
  Philos. Soc. \textbf{163} (2017), no.~3, 561--580.

\bibitem{Skeide_Shalit}
Orr Shalit and Michael Skeide, \emph{{CP}-semigroups and {D}ilations,
  {S}ub{P}roduct systems and {S}uper{P}roduct systems: {T}he {M}ultiparameter
  {C}ase and {B}eyond}, arxiv/math.OA:2003.05166.

\bibitem{Yeend}
Aidan Sims and Trent Yeend, \emph{{$C^*$}-algebras associated to product
  systems of {H}ilbert bimodules}, J. Operator Theory \textbf{64} (2010),
  no.~2, 349--376.

\bibitem{Skeide1}
Michael Skeide, \emph{Existence of {$E_0$}-semigroups for {A}rveson systems:
  making two proofs into one}, Infin. Dimens. Anal. Quantum Probab. Relat. Top.
  \textbf{9} (2006), no.~3, 373--378.

\bibitem{Skeide}
\bysame, \emph{A simple proof of the fundamental theorem about {A}rveson
  systems}, Infin. Dimens. Anal. Quantum Probab. Relat. Top. \textbf{9} (2006),
  no.~2, 305--314.

\bibitem{Skeide4}
\bysame, \emph{Classification of {$E_0$}-semigroups by product systems}, Mem.
  Amer. Math. Soc. \textbf{240} (2016), no.~1137.

\end{thebibliography}
 \bibliographystyle{amsplain}

\vspace{2.5 mm}
\noindent
{\sc S. Sundar}
(\texttt{sundarsobers@gmail.com})\\
{\footnotesize Institute of Mathematical Sciences, \\
A CI of Homi Bhabha National Institute, \\ CIT Campus, 
Taramani, Chennai, 600113, \\ Tamilnadu, India.}\\

\end{document}